\documentclass[12pt, letterpaper]{article}
\usepackage{amsmath, amsthm, amssymb, amsfonts}
\usepackage{graphicx}
\usepackage{hyperref}
\usepackage{biblatex}
\DeclareFontFamily{U}{wncy}{}
\DeclareFontShape{U}{wncy}{m}{n}{<->wncyr10}{}
\DeclareSymbolFont{mcy}{U}{wncy}{m}{n}
\DeclareMathSymbol{\Sh}{\mathord}{mcy}{"58} 

\theoremstyle{plain}
\newtheorem{theorem}{Theorem}[section]
\newtheorem{corollary}[theorem]{Corollary}
\newtheorem{definition}[theorem]{Definition}
\newtheorem{proposition}[theorem]{Proposition}
\newtheorem{remark}[theorem]{Remark}

\renewcommand{\phi}{\varphi}
\DeclareMathOperator{\sinc}{sinc}

\title{A Selection of Distributions and Their Fourier Transforms with Applications in Magnetic Resonance Imaging}
\author{Kaibo Tang}
\date{\today}
\begin{document}

\maketitle

\section{Introduction}

This note presents a rigorous introduction to a selection of distributions along with their Fourier transforms, which are commonly encountered in signal processing, and, in particular, magnetic resonance imaging (MRI). In contrast to many textbooks on the principles of MRI which place more emphasis on the signal processing aspect \cite{liang_principles_2000,nishimura_principles_2010}, this note will take a more mathematical approach. In particular, we will make explicit the underlying topological space of interest and clarify the exact sense in which these distributions and their Fourier transforms are defined. Key results presented in this note involve the Poisson summation formula and the Fourier transform of a Gaussian function via an ordinary differential equation (ODE) argument, etc.

Although the readers are expected to have prior exposure to functional analysis \cite{rudin_functional_2007} and distribution theory, this note is intended to be self-contained. To this end, we will begin with a brief review of some of the basic definitions and results.

\begin{definition}[Schwartz function]
    A Schwartz function is a function $f \in C^\infty(\mathbb{R}^n)$ such that for all $N \in \mathbb{N}$ we have $q_N(f) < \infty$ where
    \begin{equation}
        q_N(f) = \sup_{x \in \mathbb{R}^n, \vert\alpha\vert \le N} (1 + \vert x\vert)^N \vert\partial^\alpha f(x)\vert.
    \end{equation}
    We denote by $\mathcal{S}(\mathbb{R}^n) =: \mathcal{S}$ the Fr\'echet space of Schwartz functions whose topology is induced by the countable family of seminorms $\{q_N: N \in \mathbb{N}\}$.
\end{definition}

\begin{definition}[Tempered distribution]
    A tempered distribution is a continuous linear functional $u: \mathcal{S} \to \mathbb{C}$. We denote by $\mathcal{S}'(\mathbb{R}^n) =: \mathcal{S}'$ the space of tempered distributions.
\end{definition}

As usual, we use $\langle \cdot, \cdot \rangle$ to denote the dual pairing on $\mathcal{S}' \times \mathcal{S}$. The following proposition gives a convenient criterion for checking whether a linear functional is continuous.

\begin{proposition}\label{criterion}
    Let $u: \mathcal{S} \to \mathbb{C}$ be linear. Then, $u$ is continuous if and only if there exists $N, C_N$ such that
    \begin{equation}
        \vert\langle u, \phi \rangle\vert \le C_N q_N(\phi) \quad \forall \phi \in \mathcal{S}.
    \end{equation}
\end{proposition}

\begin{corollary}\label{L1}
    Every Lebesgue integrable function is a tempered distribution in the distributional sense.
\end{corollary}

\begin{proof}
    Fix $f \in L^1(\mathbb{R}^n) =: L^1$ and fix $\phi \in \mathcal{S}$. It is clear that $f: \mathcal{S} \to \mathbb{C}$ is linear. Hence, we only need to check continuity. Notice that
    \begin{equation}
        \vert \langle f, \phi \rangle \vert = \left \vert \int_{\mathbb{R}^n} f(x) \phi(x) \ dx \right \vert \le \int_{\mathbb{R}^n} \vert f(x) \vert \vert \phi(x) \vert \ dx \le \Vert f \Vert_{L^1} \sup_{x \in \mathbb{R}^n} \vert \phi(x) \vert. 
    \end{equation}
    Take $N = 1$ and $C_N = \Vert f \Vert_{L^1}$ to finish.
\end{proof}

The rest of the note is organized as follows. Section~\ref{sec2} reviews the Fourier transform, including how it is initially defined on $L^1$ and extended to all of $\mathcal{S}'$. Section~\ref{sec3} covers a selection of distributions along with their Fourier transforms: rectangular (\ref{sec31}), Gaussian (\ref{sec32}), Dirac delta (\ref{sec33}, \ref{sec34}), constant (\ref{sec33}), exponential (\ref{sec34}), and Dirac comb (\ref{sec35}).

\section{The Fourier transform}\label{sec2}

In this section, our aim is twofold. We will first outline the typical procedure of defining the Fourier transform on the space of tempered distributions $\mathcal{S}'$. Later, we will clarify how the Fourier transform defined using the mathematical convention differs from the engineering convention.

\subsection{The Fourier transform on the space of tempered distributions}

We begin by defining the Fourier transform for function in $L^1(\mathbb{R}^n) =: L^1$, which is given by an absolute convergent integral. To avoid confusion, we will denote the Fourier transform defined here as $\mathcal{F}_1$ or $\hat{\cdot}_1$. 

\begin{definition}[Fourier transform I]\label{F1}
    Let $f \in L^1$. For $\xi \in \mathbb{R}^n$, define the Fourier transform of $f$, denoted $\mathcal{F}_1f$ or $\hat{f}_1$, by
    \begin{equation}
        \mathcal{F}_1f(\xi) = \hat{f}_1(\xi) := \int_{\mathbb{R}^n} e^{-ix \cdot \xi} f(x) \ d'x, \quad \text{where } d'x := (2\pi)^{-n/2} \ dx.
    \end{equation}
\end{definition}

The following formula describes what is frequently referred to as the ``inverse Fourier transform''.\footnote{Use the Fourier inversion formula with \emph{caution} since it is valid only for a subset of functions!}

\begin{theorem}[Fourier inversion formula]
    For $g \in \mathcal{S} \subset L^1$, we have
    \begin{equation}
        g(x) = \int_{\mathbb{R}^n} e^{ix \cdot \xi} \hat{g}_1(\xi) \ d'\xi, \quad \text{where } d'\xi := (2\pi)^{-n/2} \ d\xi.
    \end{equation}
\end{theorem}

\begin{corollary}\label{homeo}
    The Fourier transform $\mathcal{F}_1$ is a homeomorphism of $\mathcal{S}$ onto itself.
\end{corollary}

We now use the dual pairing on $\mathcal{S}' \times \mathcal{S}$ to extend $\mathcal{F}_1$ to all of $\mathcal{S}'$.

\begin{proposition}
    For $u \in \mathcal{S}'$, define the linear map $\mathcal{F}_1u = \hat{u}: \mathcal{S} \to \mathbb{C}$ by
    \begin{equation}
        \langle \hat{u}, \phi \rangle := \langle u, \hat{\phi} \rangle \quad \text{for } \phi \in \mathcal{S}.
    \end{equation}
    Then, $\mathcal{F}_1u \in \mathcal{S}'$ and the map $\mathcal{F}_1: \mathcal{S}' \to \mathcal{S}'$ is a homeomorphism.
\end{proposition}

\subsection{Different conventions of the Fourier transform}

In the engineering realm, the Fourier transform is defined differently as follows. To avoid confusion, we will denote the Fourier transform defined here as $\mathcal{F}_2$ or $\hat{\cdot}_2$. 

\begin{definition}[Fourier transform II]\label{F2}
    Let $f \in L^1$. For $\xi \in \mathbb{R}^n$, define the Fourier transform of $f$, denoted $\mathcal{F}_2f$ or $\hat{f}_2$, by
    \begin{equation}
        \mathcal{F}_2f(\xi) = \hat{f}_2(\xi) := \int_{\mathbb{R}^n} e^{-2\pi ix \cdot \xi} f(x) \ dx.
    \end{equation}
\end{definition}

For the ease of comparison, we will recall the Fourier transform defined following the mathematical convention. Let $f \in L^1$. For $\xi \in \mathbb{R}^n$, the Fourier transform of $f$, defined using the mathematical convention, denoted $\mathcal{F}_2f$ or $\hat{f}_1$, is given by
\begin{equation}
    \mathcal{F}_1f(\xi) = \hat{f}_1(\xi) := \int_{\mathbb{R}^n} e^{-ix \cdot \xi} f(x) \ d'x,
\end{equation}
where $d'x := (2\pi)^{-n/2} \ dx$.

Similar to the mathematical convention, we can extend $\mathcal{F}_2$ from $\mathcal{S}$ to all of $\mathcal{S}'$ such that the Fourier transform of a distribution, e.g., the Dirac comb, can be computed. 

The mathematical and engineering conventions are reconciled by the following relationships:
\begin{align}
    \hat{f}_1(\xi) & = (2\pi)^{-n/2} \hat{f}_2\left(\frac{\xi}{2\pi}\right); \label{convert}\\
    \hat{f}_2(\xi) & = (2\pi)^{n/2} \hat{f}_1(2\pi\xi).    
\end{align}

Sometimes, the engineering convention gives a slightly more elegant expression. Hence, for the rest of the note, we will denote by $\mathcal{F}$ and $\hat{\cdot}$ the Fourier transform in the engineering convention. To obtain the expression for the mathematical convention, apply Equation~\ref{convert}.

\subsection{Properties of the Fourier transform}

Before moving to the next section, we will highlight some important properties of the Fourier transform on $\mathcal{S}$. For example, we have previously shown in Corollary~\ref{homeo} that the Fourier transform $\mathcal{F}_1:\mathcal{S} \to \mathcal{S}$ is linear, and so is $\mathcal{F}_2$.

The next few properties will turn out to be useful for some of the derivations in the next section. The proofs are immediate from Definitions~\ref{F1} and ~\ref{F2}, and are left as exercise.

\begin{proposition}\label{shifting}
    Let $\phi \in \mathcal{S}$, $x_0, \xi_0 \in \mathbb{R}^n$, and $\lambda > 0$. We have
    \begin{enumerate}
        \item (time shifting\footnote{Or spatial shifting, depending on the dimension $n$ and also what $x$ represents.}) if $h(x) = \phi(x - x_0)$, then 
        \begin{equation}\label{ts}
            \hat{h}_1(\xi) = e^{-i x_0 \cdot \xi} \hat{\phi}_1(\xi) \quad \text{and} \quad \hat{h}_2(\xi) = e^{-2\pi i x_0 \cdot \xi} \hat{\phi}_2(\xi);
        \end{equation}
        \item (frequency shifting) if $g(x) = e^{ix \cdot \xi_0} \phi(x)$, then 
        \begin{equation}
            \hat{g}_1(\xi) = \hat{\phi}_1(\xi - \xi_0);
        \end{equation}
        or if $g(x) = e^{2\pi i x \cdot \xi_0} \phi(x)$, then
        \begin{equation}\label{fs}
            \hat{g}_2(\xi) = \hat{\phi}_2(\xi - \xi_0);
        \end{equation}
        \item (time scaling\footnote{Or spatial scaling, depending on the dimension $n$ and also what $x$ represents. It is noteworthy that this scaling is \emph{isotropic} for all dimensions.}) if $p(x) = \phi(\lambda x)$, then
        \begin{equation}
            \hat{p}(\xi) = \frac{1}{\lambda^n}\hat{\phi} \left( \frac{\xi}{\lambda} \right).
        \end{equation}
    \end{enumerate}
\end{proposition}

One other property, which deserves to be mentioned separately, is the following. The proofs are also quite trivial and are left as exercise. 

\begin{proposition}[Conjugation]\label{conjugation}
    Let $\phi \in \mathcal{S}$. If $f(x) = \overline{\phi(x)}$, then
    \begin{equation}
        \hat{f}(\xi) = \overline{\hat{\phi}(-\xi)}.
    \end{equation}
    In particular, if $f$ is real, then
    \begin{equation}\label{pf}
        \hat{f}(-\xi) = \overline{\hat{f}(\xi)};
    \end{equation}
    if $f$ is imaginary, then
    \begin{equation}
        \hat{f}(-\xi) = -\overline{\hat{f}(\xi)}.
    \end{equation}
\end{proposition}

\begin{remark}
    Equation~(\ref{pf}) is the theoretical basis behind the partial Fourier techniques in MRI. Consider a perfect scenario where no phase errors occur during data collection, i.e., where the image $f$ is real. The corresponding raw data that we collect in $k$-space display conjugate symmetry. Therefore, we could essentially generate the same image by covering as little as only half of $k$-space and filling in the rest by conjugate symmetry. However, phase error occurs during acquisition and conjugate symmetry is almost always violated. One of the most significant sources of these phase errors is $B_0$ inhomogeneity. Therefore, prediction of missing data is usually performed \emph{after} phase correction, which is typically done by acquiring slightly more than half, e.g., $5/8$, of the lines in $k$-space to generate phase correction maps of $k$-space.
\end{remark}

As a corollary to Proposition~\ref{conjugation}, we have the following property about the Fourier transform of the real and imaginary parts of a function. The proofs are immediate from Proposition~\ref{conjugation} and are left as exercise.

\begin{corollary}[Real and imaginary parts]\label{re_im}
    Let $\phi \in \mathcal{S}$. If $f(x) = \Re [\phi(x)]$, then
    \begin{equation}
        \hat{f}(\xi) = \frac{1}{2}\left(\hat{f}(\xi) + \overline{\hat{f}(-\xi)}\right);
    \end{equation}
    if $f(x) = \Im[\phi(x)]$, then
    \begin{equation}
        \hat{f}(\xi) = \frac{1}{2i}\left(\hat{f}(\xi) - \overline{\hat{f}(-\xi)}\right).
    \end{equation}
\end{corollary}

\section{Some distributions and their Fourier transforms}\label{sec3}
In this section, for the sake of simplicity, we will take the dimension $n = 1$. In addition, we write $\mathcal{S}(\mathbb{R}) =: \mathcal{S}$, $\mathcal{S}'(\mathbb{R}) =: \mathcal{S}'$, and $L^1(\mathbb{R}) =: L^1$. For reference, some important distributions along with their Fourier transforms (under both conventions) are included in Table~\ref{tab}.

\renewcommand{\arraystretch}{2}

\begin{table}[ht!]
    \centering
    \begin{tabular}{lll}
        \hline
        $f(x)$ & $\hat{f}_1(\xi)$ & $\hat{f}_2(\xi)$ \\[0.5em]
        \hline
        $\sqcap(x)$ & $\displaystyle\frac{1}{\sqrt{2\pi}}\;\sinc\left(\frac{\xi}{2\pi}\right)$ & $\sinc(\xi)$ \\[0.5em]
        $e^{-\pi x^2}$ & $\displaystyle\frac{1}{\sqrt{2\pi}}\;e^{-\tfrac{\xi^2}{4\pi}}$ & $e^{-\pi \xi^2}$ \\[0.5em]
        $1$ & $\sqrt{2\pi}\;\delta(\xi)$ & $\delta(\xi)$ \\[0.5em]
        $\delta(x)$ & $\displaystyle\frac{1}{\sqrt{2\pi}}$ & 1 \\[0.5em]
        $e^{2\pi i x\xi_0 }$ & $\sqrt{2\pi}\;\delta(\xi - 2\pi \xi_0)$ & $\delta(\xi - \xi_0)$ \\[0.5em]
        $\delta(x - x_0)$ & $\displaystyle\frac{1}{\sqrt{2\pi}}\;e^{-ix_0\xi}$ & $e^{-2\pi i x_0\xi}$ \\[0.5em]
        $\cos(2\pi x \xi_0)$ & $\displaystyle\frac{\sqrt{2\pi}}{2}\left[\delta(\xi - 2\pi \xi_0) + \delta(\xi + 2\pi \xi_0)\right]$ & $\displaystyle\frac{1}{2}\left[\delta(\xi + \xi_0) + \delta(\xi - \xi_0)\right]$ \\[0.5em]
        $\sin(2\pi x \xi_0)$ & $\displaystyle\frac{\sqrt{2\pi}}{2i}\left[\delta(\xi - 2\pi \xi_0) - \delta(\xi + 2\pi \xi_0)\right]$ & $\displaystyle\frac{1}{2i}\left[\delta(\xi + \xi_0) - \delta(\xi - \xi_0)\right]$ \\[0.5em]
        $\Sh(x)$ & $\displaystyle\frac{1}{\sqrt{2\pi}}\;\Sh\left(\frac{\xi}{2\pi}\right)$ & $\Sh(\xi)$ \\[0.5em]
        \hline
    \end{tabular}
    \caption{Some distributions and their Fourier transforms.}
    \label{tab}
\end{table}

\subsection{Rectangular function}\label{sec31}

\begin{definition}[Rectangular function]
    Define the rectangular function, denoted $\sqcap$, by
    \begin{equation}
        \sqcap(x) := 
        \begin{cases}
            1 & \text{if } \vert x \vert \le 1/2 \\
            0 & \text{otherwise}
        \end{cases}.
    \end{equation}
\end{definition}

Integrating $\sqcap$ over the real line shows that $\Vert \sqcap \Vert_{L^1} = 1$. Hence, we see that $\sqcap \in L^1$ and, by Corollary~\ref{L1}, that $\sqcap: \mathcal{S} \to \mathbb{C}$ is a tempered distribution.

\begin{definition}[Sinc function]
    Define the (normalized)\footnote{The unnormalized sinc function is defined for $x \neq 0$ by $\sinc(x) = \frac{\sin x}{x}$ and defined for $x = 0$ similar to the normalized version.} sinc function for $x \neq 0$ by
    \begin{equation}
        \sinc(x) := \frac{\sin(\pi x)}{\pi x}.
    \end{equation}
    Additionally, define $\sinc(0)$ to be the limit of $\sinc(x)$ as $x \to 0$, i.e.,
    \begin{equation}
        \sinc(0) := \lim_{x \to 0} \frac{\sin (\pi x)}{\pi x} = 1.
    \end{equation}
\end{definition}

\begin{proposition}
    The Fourier transform of the rectangular function $\sqcap \in \mathcal{S}'$, denoted $\hat{\sqcap} \in \mathcal{S}'$, satisfies $\hat{\sqcap} = \sinc$.
\end{proposition}

\begin{proof}
    Since $\sqcap \in L^1$, we prove by direct computation, i.e.,
    \begin{align}
        \hat{\sqcap}(\xi) = \int_{-\infty}^\infty \sqcap(x) e^{-2\pi ix\xi} \ dx = \int_{-1/2}^{1/2} e^{-2\pi ix\xi} \ dx & = \left[ \frac{e^{-2\pi ix\xi}}{-2\pi i\xi} \right]_{x = -1/2}^{x = 1/2} 
        \\ & = \frac{-2i\sin(\pi \xi)}{-2\pi i \xi} 
        \\ & = \frac{\sin (\pi \xi)}{\pi \xi} = \sinc(\xi).
    \end{align}
\end{proof}

\subsection{Gaussian function}\label{sec32}

We consider a particular example of Gaussian function of the form 
\begin{equation}
    g(x) = e^{-\pi x^2}.
\end{equation}
Integrating $g$ over the real line shows that $\Vert g \Vert_{L^1} = 1$. Hence, we see that $g \in L^1$ and, by Corollary~\ref{L1}, that $g: \mathcal{S} \to \mathbb{C}$ is a tempered distribution.

\begin{proposition}
    The Fourier transform of 
    \begin{equation}\label{g}
        g(x) = e^{-\pi x^2}
    \end{equation}
    is given by
    \begin{equation}
        \hat{g}(\xi) = e^{-\pi \xi^2}.
    \end{equation}
\end{proposition}

\begin{proof}
    We show that $g$ and $\hat{g}$ satisfy the same initial value problem for which the uniqueness theorem applies, and $g = \hat{g}$ follows immediately.
    \begin{enumerate}
        \item ($g$). Differentiating both sides of (\ref{g}), we see that $g$ is a solution of the initial value problem given by
        \begin{equation}\label{ivp}
            f'(x) = -2\pi x f(x), f(0) = 1.
        \end{equation}
        \item ($\hat{g}$). The Fourier transform of $g$ is given by
        \begin{equation}
            \hat{g}(\xi) = \int_{-\infty}^\infty g(x) e^{-2\pi ix\xi} \ dx = \int_{-\infty}^\infty e^{-\pi x^2} e^{-2\pi ix\xi} \ dx.
        \end{equation}
        Differentiating the integrand, we see that
        \begin{equation}
            \left \vert \frac{d}{d\xi} [e^{-\pi x^2} e^{-2\pi ix\xi}] \right \vert = \vert e^{-\pi x^2} (-2\pi ix) e^{-2\pi ix\xi} \vert = 2\pi \vert x \vert e^{-\pi x^2},
        \end{equation}
         which is Lebesgue integrable over $\mathbb{R}$. Hence, the dominated convergence theorem applies and we may differentiate under the integral sign, in which case, we have
         \begin{equation}
             \hat{g}'(\xi) = \int_{-\infty}^\infty e^{-\pi x^2} (-2\pi ix) e^{-2\pi ix\xi} \ dx = i \int_{-\infty}^\infty (-2\pi x)e^{-\pi x^2} e^{-2\pi ix\xi} \ dx.
         \end{equation}
         Integrating by parts and noticing that the boundary term vanishes, we get 
         \begin{align}
             \hat{g}'(\xi) & = -i \int_{-\infty}^\infty e^{-\pi x^2} (-2\pi i\xi) e^{-2\pi ix\xi} \ dx \\
             & = -2\pi\xi i \int_{-\infty}^\infty e^{-\pi x^2} e^{-2\pi ix\xi} \ dx \\
             & = -2\pi\xi \hat{g}(\xi).
         \end{align}
         Finally, notice that 
         \begin{equation}
             \hat{g}(0) = \int_{-\infty}^\infty g(x) \ dx = 1.
         \end{equation}
         Hence, $\hat{g}$ is also a solution of the initial value problem in (\ref{ivp}).
    \end{enumerate}
\end{proof}

\begin{remark}
    There are many ways\footnote{\href{https://math.stackexchange.com/questions/270566/how-to-calculate-the-fourier-transform-of-a-gaussian-function}{https://math.stackexchange.com/questions/270566/how-to-calculate-the-fourier-transform-of-a-gaussian-function}} of computing the Fourier transform of a Gaussian function. But the ODE proof given above is one of the most elementary proofs that I know of since it does not involve invoking theorems from complex analysis. In addition, the proof above uses some of the most important conclusions from real analysis and ODE theory, namely, the fundamental existence and uniqueness theorem, Lebesgue's dominated convergence theorem, and integration by parts.
\end{remark}

\subsection{Dirac delta and constant function}\label{sec33}

In this section, we will see that the Dirac delta $\delta$ and the constant function $1$ are related by Fourier transform. As you will notice before long, one direction ($\hat{\delta} = 1$) is quite trivial while the other direction ($\hat{1} = \delta$), despite still being quite trivial, requires some additional tricks. 

We start by defining the Dirac delta and checking both $1$ and $\delta$ are tempered distributions.

\begin{proposition}\label{const}
    The constant function $1: \mathcal{S} \to \mathbb{C}$ is a tempered distribution.
\end{proposition}

\begin{proof}
    Fix $\phi \in \mathcal{S}$. It is clear that $1: \mathcal{S} \to \mathbb{C}$ is linear. Hence, we only need to check continuity. Notice that 
    \begin{align}
        \vert \langle 1, \phi \rangle \vert = \left \vert \int_{-\infty}^\infty \phi(x) \ dx \right \vert & \le \int_{-\infty}^\infty \vert \phi(x) \vert \ dx \\
        & = \int_{-\infty}^\infty \frac{1}{(1+|x|)^2}(1+|x|)^2 \vert \phi(x) \vert \ dx \\
        & = \left ( \int_{-\infty}^\infty \frac{1}{(1+|x|)^2} \ dx \right ) \sup_{x \in \mathbb{R}} [(1+|x|)^2 \vert \phi(x) \vert] \\
        & = 2 \sup_{x \in \mathbb{R}} [(1+|x|)^2 \vert \phi(x) \vert].
    \end{align}
    Take $N = 2$ and $C_N = 2$ in Proposition~\ref{criterion} to finish.
\end{proof}

Here, we define the Dirac delta, also known as the unit impulse.

\begin{definition}[Dirac delta]
    Define the Dirac delta, denoted $\delta: \mathcal{S} \to \mathbb{C}$, by
    \begin{equation}
        \langle \delta, \phi \rangle = \phi(0) \quad \text{for } \phi \in \mathcal{S}.
    \end{equation}
\end{definition}

\begin{proposition}
    The Dirac delta $\delta: \mathcal{S} \to \mathbb{C}$ is a tempered distribution.
\end{proposition}

\begin{proof}
    Fix $\phi \in \mathcal{S}$. It is clear that $f: \mathcal{S} \to \mathbb{C}$ is linear. Hence, we only need to check continuity. Notice that
    \begin{equation}
        \vert \langle \delta, \phi \rangle \vert = \vert \phi(0) \vert \le \sup_{x \in \mathbb{R}} \vert \phi(x) \vert.
    \end{equation}
    Take $N = 1$ and $C_N = 1$ in Proposition~\ref{criterion} to finish.
\end{proof}

We will tackle the less trivial direction first, i.e., $\hat{1} = \delta$.

\begin{proposition}
    The Fourier transform of the constant function $1$ is $\delta$.
\end{proposition}

\begin{proof}
    For $\phi \in \mathcal{S}$, we have
    \begin{equation}
        \langle \hat{1}, \phi \rangle = \langle 1, \hat{\phi} \rangle = \int_{-\infty}^\infty \hat{\phi}(\xi) \ d\xi.
    \end{equation}
    Since $\phi \in \mathcal{S}$, the Fourier inversion formula applies, in which case,
    \begin{equation}
        \langle \hat{1}, \phi \rangle = \int_{-\infty}^\infty e^{2\pi i0\xi} \hat{\phi}(\xi) \ d\xi = \phi(0) = \langle \delta, \phi \rangle.
    \end{equation}
\end{proof}

Now, the other direction is easier to show.

\begin{proposition}
    The Fourier transform of the Dirac delta $\delta$ is $1$.
\end{proposition}

\begin{proof}
    Notice that
    \begin{equation}
        \langle \hat{\delta}, \phi \rangle = \langle \delta, \hat{\phi} \rangle = \hat{\phi}(0) = \int_{-\infty}^\infty e^{-2\pi ix0} \phi(x) \ dx = \int_{-\infty}^\infty \phi(x) \ dx = \langle 1, \phi \rangle.
    \end{equation}
\end{proof}

\subsection{Dirac delta and Exponential function}\label{sec34}

In the previous section, we demonstrated that the constant function and the Dirac delta are Fourier transforms of each other. In this section, we will first apply the time and frequency shifting properties of the Fourier transform in Proposition~\ref{shifting} to obtain two more Fourier transform pairs. Then, we will apply Corollary~\ref{re_im} to obtain the Fourier transforms of the cosine and sine functions.

An argument similar to the proof of Proposition~\ref{const} (how?) shows that exponential functions of the form $e^{2\pi i x \xi_0}$ defines a tempered distribution. The following proposition demonstrates that the exponential function $e^{2\pi i x \xi_0}$ and the Dirac delta are Fourier transforms of each other. Proofs for both equations are trivial. Take $\phi = \delta$ in Equations~(\ref{fs}) and (\ref{ts}) and the results are immediate.

\begin{proposition}
    \quad
    \begin{enumerate}
        \item If $f(x) = e^{2\pi i x \xi_0}$, then
        \begin{equation}
            \hat{f}(\xi) = \delta(\xi - \xi_0);
        \end{equation}
        \item if $f(x) = \delta(x - x_0)$, then
        \begin{equation}
            \hat{f}(\xi) = e^{-2\pi i\xi x_0}.
        \end{equation}
    \end{enumerate}
\end{proposition}

Now, notice that 
\begin{equation}
     \cos(2 \pi x \xi_0) = \Re[e^{2\pi i x \xi_0}] \quad \text{and} \quad \sin(2 \pi x \xi_0) = \Im[e^{2\pi i x \xi_0}].
\end{equation}
So, Corollary~\ref{re_im} applies. With some algebraic manipulation, we obtain the following corollary.

\begin{corollary}
    \quad
    \begin{enumerate}
        \item If $f(x) = \cos(2\pi x \xi_0)$, then
        \begin{equation}
            \frac{1}{2}\left[\delta(\xi + \xi_0) + \delta(\xi - \xi_0)\right];
        \end{equation}
        \item if $f(x) = \sin(2\pi x \xi_0)$, then
        \begin{equation}
            \frac{1}{2i}\left[\delta(\xi + \xi_0) - \delta(\xi - \xi_0)\right].
        \end{equation}
    \end{enumerate}
\end{corollary}

\subsection{Dirac comb}\label{sec35}

The Dirac comb is a powerful distributional tool. It enables us to model sampling of a function. For example, in MRI, we are usually concerned with sampling in the frequency domain, i.e., $k$-space, where the acquired raw data in $k$-space can be modeled as a pointwise product between the underlying function in $k$-space and the Dirac comb with an appropriately chosen sampling period. Additional applications of the Dirac comb involves aliasing and the discrete Fourier transform (DFT) among many others.

\begin{definition}[Dirac comb]
    Define the Dirac comb, denoted $\Sh$, by
    \begin{equation}
        \Sh(x) := \sum_{n = -\infty}^\infty \delta(x - n).
    \end{equation}
\end{definition}

\begin{proposition}
    The Dirac comb $\Sh: \mathcal{S} \to \mathbb{C}$ is a tempered distribution. 
\end{proposition}

\begin{proof}
    It is clear that $\Sh: \mathcal{S} \to \mathbb{C}$ is linear. We only need to check continuity using Proposition~\ref{criterion}. Indeed, for any fixed $\phi \in \mathcal{S}$, we have
    \begin{align}
        \vert\langle \Sh, \phi \rangle\vert = \left\vert \sum_{n = -\infty}^\infty \phi(n) \right\vert & = \left\vert \sum_{n = -\infty}^\infty \frac{1}{(1 + \vert n\vert)^2} (1 + \vert n\vert)^2 \phi(n) \right\vert \\
        & \le \left( \sum_{n = -\infty}^\infty \frac{1}{(1 + \vert n\vert)^2} \right) \sup_{x \in \mathbb{R}} \vert(1 + \vert x\vert)^2 \phi(x)\vert.
    \end{align}
    Finally, take $N = 2$ and 
    \begin{equation}
        C_N = \sum_{n = -\infty}^\infty \frac{1}{(1 + \vert n\vert)^2} < \infty,
    \end{equation}
    and apply Proposition~\ref{criterion} to finish.
\end{proof}

Now, we will state and prove the Poisson summation formula, which is a powerful tool and can be used to derive the Nyquist–Shannon sampling theorem. And we will derive the Fourier transform of the Dirac comb as a corollary of the Poisson summation formula.

\begin{proposition}[Poisson summation formula]\label{psf}
    If $\phi \in \mathcal{S}$, then
    \begin{equation}
        \sum_{n = -\infty}^\infty \phi(x + n) = \sum_{n = -\infty}^\infty \hat{\phi}(n)e^{2\pi inx}.
    \end{equation}
    In particular, setting $x = 0$,
    \begin{equation}
        \sum_{n = -\infty}^\infty \phi(n) = \sum_{n = -\infty}^\infty \hat{\phi}(n).
    \end{equation}
\end{proposition}

\begin{proof}
    Define 
    \begin{equation}\label{sum}
        f(x) := \sum_{n = -\infty}^\infty \phi(x + n).
    \end{equation}
    It is clear that $f$ is a periodic function with period $1$ (check!) and the infinite series converge uniformly on compact sets (check!), and, in particular, on $[0, 1]$. 
    
    It suffices to check that $\hat{\phi}(n)$ are the coefficients in the Fourier series expansion of $f$. Indeed, the coefficients in the Fourier series expansion of $f$ are given by
    \begin{equation}
        c_m = \int_0^1 f(x) e^{-2\pi imx} \ dx = \int_0^1 \sum_{n = -\infty}^\infty \phi(x + n) e^{-2\pi imx} \ dx.
    \end{equation}
    Since the infinite sum in (\ref{sum}) converges uniformly on $[0,1]$, we can interchange limit processes,\footnote{That is, interchanging the integral and the infinite sum.} which gives
    \begin{align}
        c_m = \sum_{n = -\infty}^\infty \int_0^1 \phi(x + n) e^{-2\pi imx} \ dx & = \sum_{n = -\infty}^\infty \int_{n}^{n+1} \phi(x) e^{-2\pi imx} \ dx \\
        & = \int_{-\infty}^\infty \phi(x) e^{-2\pi imx} \ dx = \hat{\phi}(m).
    \end{align}
\end{proof}

\begin{corollary}
    The Fourier transform of the Dirac comb $\Sh$ is $\Sh$.
\end{corollary}

\begin{proof}
    Fix $\phi \in \mathcal{S}$. Proposition~\ref{psf} implies that
    \begin{equation}
        \langle \hat{\Sh}, \phi \rangle = \langle \Sh, \hat{\phi} \rangle = \sum_{n = -\infty}^\infty \hat{\phi}(n) = \sum_{n = -\infty}^\infty \phi(n) = \langle \Sh, \phi \rangle.
    \end{equation}
\end{proof}

\printbibliography

@book{rudin_functional_2007,
	address = {Boston, Mass.},
	edition = {2 ed.},
	series = {International {Series} in {Pure} and {Applied} {Mathematics}},
	title = {Functional {Analysis}},
	isbn = {978-0-07-054236-5},
	publisher = {McGraw-Hill},
	author = {Rudin, Walter},
	year = {2007},
}

@book{liang_principles_2000,
	address = {New York},
	series = {{IEEE} {Press} {Series} in {Biomedical} {Engineering}},
	title = {Principles of {Magnetic} {Resonance} {Imaging}: {A} {Signal} {Processing} {Perspective}},
	isbn = {978-0-7803-4723-6},
	shorttitle = {Principles of {Magnetic} {Resonance} {Imaging}},
	publisher = {IEEE Press},
	author = {Liang, Zhi-Pei and Lauterbur, Paul C.},
	year = {2000},
}

@book{nishimura_principles_2010,
	title = {Principles of {Magnetic} {Resonance} {Imaging}},
	publisher = {Lulu},
	author = {Nishimura, Dwight G.},
	month = jan,
	year = {2010},
}

\end{document}